\documentclass[a4paper,11pt]{amsart}

\usepackage[utf8]{inputenc} 
\usepackage[T1]{fontenc}      
\usepackage{geometry}         
\usepackage[english]{babel}  
\usepackage{amsmath}                            
\usepackage{amsfonts}
\usepackage{amsthm}                             
  
\usepackage{amscd}
\usepackage{hyperref} 
\usepackage{yfonts}

   \usepackage{amssymb}
   \usepackage{verbatim}
   \newif\ifpdf
   \ifpdf
     \usepackage[pdftex]{graphicx}
     \usepackage[pdftex]{hyperref}
   \el

     \usepackage{graphicx}
   \fi

\usepackage{dsfont}\let\mathbb\mathds
\usepackage[all]{xy}
   
\usepackage{MnSymbol}

\newcommand{\N}{\mathbb{N}} 
 
\newcommand{\C}{\mathbb{C}}

\newcommand{\teich}{\mathrm{Teich}}

\newtheorem{theo}{Theorem}  
   
\newtheorem{coro}{Corollary}      
\newtheorem{defi}{Definition} 
\newtheorem{lem}{Lemma} 
\newtheorem{rem}{Remark}

\newtheorem{step}{Step}

\newtheorem*{theoA}{Theorem A}
\newtheorem*{theoB}{Theorem B}
\newtheorem*{theoC}{Theorem C}

\newcommand{\rs}{\mathbb{P}^1}

\newcommand{\limn}{\lim_{n \rightarrow \infty}}
\newcommand{\id}{\mathrm{Id}}

\newcommand{\mcal}{\mathcal{M}}

\newcommand{\fatou}{\mathcal{F}}
\newcommand{\julia}{\mathcal{J}}

\newcommand{\ptwo}{\mathbb{P}^2}
\newcommand{\pk}{\mathbb{P}^k}

\newcommand{\card}{\mathrm{card\,}}

\newcommand{\pcal}{\mathcal{P}}

\newcommand{\jac}{\mathrm{Jac}}
\newcommand{\lcm}{\mathrm{lcm}}

\newcommand{\struct}{\textbf{s}}
\newcommand{\pf}{\mathcal{P}}

\newcommand{\crit}{\mathcal{C}}
\newcommand{\sing}{\mathrm{Sing}}
\newcommand{\cv}{\mathrm{CV}}
\newcommand{\reg}{\mathrm{Reg}}
\newcommand{\codim}{\mathrm{codim}\,}

\title{Dynamics of post-critically finite maps in higher dimension}
\author{Matthieu Astorg}

\email{matthieu.astorg@univ-orleans.fr}
\address{Université d'Orléans, Collegium Sciences et Techniques \\
	Bâtiment de mathématiques - Rue de Chartres\\
	B.P. 6759 - 45067 Orléans Cedex 2 \\
	FRANCE}                        
                        
\begin{document}
	\begin{abstract}
		We study the dynamics of post-critically finite endomorphisms of $\pk(\C)$. We prove that post-critically finite
		endomorphisms are always post-critically finite all the way down under a mild regularity condition on the post-critical set.
		We study the eigenvalues of periodic points of post-critically finite endomorphisms.
		Then, under a weak transversality condition and assuming Kobayashi hyperbolicity of the complement of the post-critical 
		set, we prove that the only possible Fatou components are super-attracting basins, thus partially extending to 
		any dimension a result of Fornaess-Sibony and Rong holding in the case $k=2$.
	\end{abstract}

	\maketitle

	\section{Introduction}

	This paper deals with the dynamics of some endomorphisms of the complex projective space $\pk$,
	called post-critically finite (PCF). 
	An endomorphism $f: \pk\to \pk$ is post-critically finite if
	$$\pcal(f,\pk):=\bigcup_{n \geq 1} f^n(\crit(f))$$
	is algebraic, where $\crit(f, \pk):=\{z \in \pk, Df(z) \text{ is not invertible}\}$. 
	The set $\crit(f, \pk)$ is called the critical set of $f$, and the set $\pf(f,\pk)$ is called the post-critical set 
	of $f$. 
	We will often consider restrictions of $f$ to an invariant subvariety $L \subset \pk$: in this case,
	$\crit(f,L)$ and $\pf(f,L)$ will refer to the critical and post-critical sets of the restriction map $f: L \to L$.
	When there is no ambiguity, we will just denote those sets by $\pf(f)$ and $\crit(f)$.

	In the case where $k=1$, this class of endomorphism has been particularly studied, mainly because of 
	an important classification theorem due to Thurston
	(\cite{thurston1985combinatorics}). The dynamics of those maps is also well understood.
	
	Thurston's theorem does not hold in higher dimension ($k>1$), but post-critically finite 
	maps still are of interest, for example because of a construction due to Koch 
	(\cite{koch2013teichmuller}).
	The dynamics of PCF maps has been investigated by several authors. 
	In \cite{fornaess1992critically}, Fornaess and Sibony studied the dynamics of two specific
	examples of PCF endomorphisms of $\ptwo$.
	In \cite{fornaess1994complex}, the same authors considered more generally 
	PCF endomorphisms
	of $\pk$ and studied the case $k=2$. They notably proved that if $\ptwo \backslash \pf(f)$ is Kobayashi hyperbolic, where $f$ is a PCF
	endomorphism of $\ptwo$, then the only Fatou components of $f$ are basins of superattracting cycles.
	As Jonsson notes in \cite{jonsson1998some}, the case of dimension 2 is still special, as in this case the post-critical set 
	has dimension 1, and it can be proved that the normalization of post-critical components must be either 
	tori or projective lines. In still higher dimension ($k \geq 3$), no such classification is readily available.

	In \cite{ueda1998critical}, Ueda defined the notion of an endomorphism $f: \pk \to \pk$ that is 
	PCF of order $m$: informally, a PCF endomorphism is PCF of order 1. Then every critical component
	is eventually mapped to a periodic post-critical component, and one can consider the first return maps 
	induced by restriction of iterates of $f$ to those periodic components. If all of them are themselves
	post-critically finite, then $f$ is post-critically finite of order 2. One can then define inductively 
	what it means to be post-critically finite of order $m$, with $m \leq k$ (see Definition \ref{def:pcfatwd} for 
	a precise definition). 
	Post-critically finite endomorphisms of $\pk$ of order $k$ are somewhat colloquially referred to as 
	\emph{post-critically finite all the way down} (the authors of \cite{gauthier2016symmetrization} use the 
	term \emph{strongly post-critically finite}). 
	
	Those notions lead to defining post-critical finiteness for
	self-branched covers $f: X \to X$, where $X$ is a regular subvariety of $\pk$, or more generally a complex manifold.
	Such maps may have empty post-critical set (if they are unbranched self-covers); 
	in this case, most authors do not consider them to be post-critically finite.
	We shall call them \emph{weakly post-critically finite}, and we say that $f$ is weakly post-critically finite all the way 
	down if every first return map on periodic post-critical components is weakly post-critically finite.
	Again, see Definition \ref{def:pcfatwd} for full details.

	In \cite{jonsson1998some}, Jonsson remarks that PCF endomorphisms of $\ptwo$ are always post-critically finite 
	all the way down. In \cite{gauthier2016symmetrization}, the authors prove that the same is true for a very specific
	subclass of PCF endomorphisms of $\pk$, and ask if the same holds for any PCF endomorphism of $\pk$.
	They also give examples of rational maps $f: \pk \dashedrightarrow \pk$ that are PCF but not PCF all the way down.
	A key point is that those examples are not endomorphisms of $\pk$ as they have non-empty indeterminacy set.	
	We give a partial answer to this question, using a purely complex-analytical approach:
	
	\begin{theoA}
		Let $X$ be a complex manifold and let $f: X \to X$ be a PCF branched covering. Let 
		$(L_i)_{i \in I}$ be the periodic irreducible components of $\pf(f,X)$ and assume that
		for any $J \subset I$, $\bigcap_{j \in J} L_j$ is smooth. Then $f$ is weakly
		PCF all the way down.
	\end{theoA}
	
	In fact, we prove a slightly more precise statement, see Theorem \ref{th:pcfatwdcover}.\\

	An important topic in dynamical systems is the study of local dynamics near periodic points. 
	In holomorphic dynamics, the local dynamics of a periodic point $p$ of period $m$ is 
	governed by the eigenvalues of $Df^m(p)$; abusing terminology, we call them the eigenvalues of $p$.
	If $\lambda \in \C$ is such an eigenvalue, we say that $\lambda$ is:
	\begin{enumerate}
		\item attracting if $|\lambda|<1$
		\item repelling if $|\lambda|>1$
		\item super-attracting if $\lambda=0$
		\item neutral if $|\lambda|=1$: more precisely, $\lambda$ is 
		\begin{enumerate}
			\item[a.] parabolic if $\lambda$ is a root of unity
			\item[b.] irrationally neutral otherwise.
		\end{enumerate}
		
	\end{enumerate}
	
	If $p$ is a periodic point such that all of its eigenvalues are attracting (respectively repelling), we say that $p$ itself is attracting
	(respectively repelling).

	In \cite{jonsson1998some}, Jonsson studied PCF endomorphisms of $\ptwo$ such that every critical component
	is strictly preperiodic, and such that the first return maps of the periodic components of $\pf(f)$ also have only
	strictly preperiodic critical points: such PCF endomorphisms are called \emph{2-critically finite}.
	He proved that for such maps, every periodic point is repelling, and that repelling periodic points are dense in $\ptwo$. 
	This result, together with the classical fact that PCF maps on $\rs$ can only have repelling or super-attracting periodic points, 
	raises the question of knowing which eigenvalues a periodic point of a PCF endomorphism on $\pk$ may have.
	We prove the following:

	\begin{theoB}
		Let $f: \pk \to \pk$ be a PCF endomorphism such that the periodic irreducible components of
		$\pcal(f,\pk)$ are weakly transverse (see definition \ref{def:weaklytransverse}).
		Let $p$ be a periodic point of period $l$ for $f$. Then 
		\begin{enumerate}
			\item $p$ does not have any parabolic or non-zero attracting eigenvalue
			\item either $p$ has at least one repelling eigenvalue, or else all eigenvalues are super-attracting
			(ie $Df^l(p)$ is nilpotent)
			\item if $p$ doesn't have any super-attracting eigenvalue, then $p$ is Lyapunov unstable
			and at least one of its eigenvalues is repelling.
		\end{enumerate}
	\end{theoB}
	
	See Definition \ref{def:lyapunstable} for a definition of the notion of Lyapunov unstability.
	We could not exclude the possibility of a periodic point having some irrationnally neutral eigenvalues. To the best
	of our knowledge, no example of PCF endomorphism of $\pk$ with a periodic point having such an eigenvalue is known
	(it is known that it cannot occur in dimension 1).
	It would be interesting to know whether such examples exist or not. \\

	Lastly, we study the Fatou dynamics of PCF endomorphisms of $\pk$. 
	For an endomorphism $f: \pk \to \pk$, the Fatou set is the largest open subset of 
	$\pk$ on which the iterates $\{f^n, n \in \N \}$ form a normal family. We denote it by $\fatou(f)$.
	Its complement is the Julia set $\julia(f)=\julia_1(f)$. Alternatively, one can construct
	a naturel $(1,1)$ positive closed current $T_f$ called the Green current of $f$, whose
	support is exactly $\julia(f)$, and define a stratification of Julia sets 
	$\julia_m(f)$ (with $1 \leq m \leq k$) as the support of $\bigwedge_{i=1}^m T_f$.
	As $m$ increases, $\julia_m(f)$ decreases and the dynamics on $\julia_m(f)$ become 
	more chaotic in some sense. In \cite{de2005phenomene}, de Thélin analysed the structure 
	of $T_f$ on $\julia_1(f) \backslash \julia_2(f)$ for PCF endomorphisms of $\ptwo$: he 
	proved that in that case $T_f$ is laminar, something that is not true for 
	general endomorphisms of $\ptwo$. Our next result is focuses on the dynamics 
	in the Fatou set, outside of $\julia_1(f)$.
	
	 We call Fatou component a connected component of the Fatou set.
	For more background on Fatou-Julia theory, see
	for example \cite{dinh2010dynamics}.

	In \cite{fornaess1994complex}, Fornaess and Sibony proved that if $f: \ptwo \to \ptwo$ is PCF and 
	$\ptwo \backslash \pf(f)$ is Kobayashi hyperbolic, then the only Fatou components 
	of $f$ are basins of super-attracting cycles. Using methods from 
	\cite{ueda1998critical}, Rong was able in \cite{rong2008fatou} to remove the assumption of Kobayashi 
	hyperbolicity, still in the case of dimension 2.
	
	We give another partial generalization of this result, that is available in any dimension
	but requires Kobayashi hyperbolicity and a mild transversality condition on the irreducible
	components of the post-critical set:

	\begin{theoC}
		Let $f: \pk \to \pk$ be a PCF endomorphism, and let $L_i, i  \in I$ be the irreducible components
		of its post-critical set $\pf(f,\pk)$. Assume that
		\begin{enumerate}
			\item $(L_i)_{i \in I}$ is weakly transverse
			\item $\pk \backslash\pf(f,\pk)$ is Kobayashi hyperbolic.
		\end{enumerate}
		Then the Fatou set of $f$ is either empty or a finite union of basins of super-attracting cycles,
		whose points are zero-dimensional intersections of components $L_i$.
	\end{theoC}

	Note that this in particular implies that there are only finitely many periodic Fatou components,
	something that it always true for endomorphisms of $\rs$ but not necessarily so for endomorphisms of 
	$\pk$ with $k \geq 2$, and 
	that every Fatou component is preperiodic (the first examples of endomorphisms of $\pk$ having 
	a non-preperiodic Fatou component have been constructed in \cite{astorg2014two}).\\

	Finally, Koch has given in \cite{koch2013teichmuller} a systematic method for constructing 
	examples of PCF endomorphisms of $\pk$. This is interesting, as constructing non-trivial PCF endomorphisms
	is not easy. Indeed, most examples in the literature can be recovered from this construction.
	The construction starts with a purely topological object called 
	a \emph{topological polynomial}, and gives a holomorphic PCF endomorphism of $\pk$ whose dynamics
	is related to the action of that topological polynomial on a Teichmüller space. 
	In Section \ref{sec:koch}, we recall Koch's results and we show that Theorems A, B and C
	apply to her examples.

	\subsection*{Acknowledgements}I would like to thank Sarah Koch for introducing me to these questions, and for many helpful discussions.
	I also thank Xavier Buff for inviting me to the University of Toulouse and for helpful discussions. 
	This article was written almost entirely while at the University of Michigan.

	\subsection*{Outline}In Section 2, we introduce a weak notion of transversality and we prove some lemmas
	that will be later required. In Section 3, we define properly the notions of weakly PCF, PCF all the way down 
	and weakly PCF all the way down, and we prove Theorem A. In Section 4, we prove a slightly generalized version of a theorem 
	of Ueda on the absence of rotation domains for PCF maps, that will be required for the proof of Theorems B and C.
	In Section 5, we prove Theorem B, and in Section 6, we prove Theorem C. Finally, in Section 7, we discuss how our 
	results apply to the examples from \cite{koch2013teichmuller}.
	
	\subsection*{Notations}
	In the rest of the article, if $L$ is an analytic subset of a complex manifold $X$, we will denote 
	the singular part of $L$ by $\sing(L)$ and its regular part by $ \reg(L)$.

	\section{Weak transversality and Kobayashi metric}

	\begin{defi}\label{def:weaklytransverse}
		Let $X$ be a complex manifold, and let $L_i, i \in I$ be a finite set of irreducible
		hypersurfaces of $X$. We say that the $(L_i)_{i \in I}$ are \emph{weakly transverse}
		if the following holds: for any $x_0 \in \bigcup_{i \in I} L_i$, if $J=\{i \in I, x_0 \in L_i\}$ and
		if locally near $x_0$ the sets $L_i$ can be written $L_i=\{P_i=0\}$ (where the $P_i$ 
		are holomorphic functions defined locally near $x_0$), then $x \mapsto (P_j(x))_{j \in J}$ has constant rank near $x_0$.
	\end{defi}
	
	Note that in the particular case where $x \mapsto (P_j(x))_{j \in J}$ is surjective,
	the Implicit Function Theorem implies that $\bigcup_{i \in I} L_i$ has simple normal crossings.
	This is a weaker requirement: for example any finite collection of hyperplanes in a projective space 
	will be weakly transverse in the sense of the above definition.
	
	A straightforward consequence of the Constant Rank Theorem is that if $(L_i)_{i \in I}$
	are weakly transverse, then for any $J \subset I$, $\bigcap_{j \in J} L_j$ is a smooth submanifold
	(though not necessarily of codimension $\card J$), of tangent space at $x_0$ given 
	by $\bigcap_{j \in J} T_{x_0} L_j$.
	
	Let us begin by proving two lemmas that will be used in later sections.

	\begin{lem}\label{lem:kobtan}
		Let $X$ be a complex manifold, and let $(L_i)_{i \in I}$ be weakly transverse. 
		Let $x_0 \in \bigcup_{i \in I} L_i$ and $J=\{i \in I, x_0 \in L_i\}$. Let 
		$v \in \bigcap_{j \in J} T_{x_0} L_j$. There exists a holomorphic vector field $\xi$
		defined in a neighborhood $U$ of $x_0$, with $\xi(x_0)=v$, and such that 
		$\xi$ is uniformly bounded in the Kobayashi pseudometric of $X \backslash\bigcup_{i \in I} L_i$ on
		$U \backslash \bigcup_{i \in I} L_i$.
	\end{lem}
	
	\begin{proof}
		For $j \in J$, let $P_j$ be a holomorphic function defined in the neighborhood of $x_0$ 
		such that locally near $x_0$, $L_j=\{P_j=0\}$.
		Let $\Phi: x \mapsto (P_j(x))_{j \in J}$. Then by weak transversality, $\Phi$ has constant rank,
		so in local coordinates it may be written 
		$\Phi(z_1, \ldots, z_k)=(z_1, \ldots, z_m, 0, \ldots ,0)$ (here $k=\dim X$ and 
		$m=\mathrm{codim\,} \bigcap_{j \in J} L_j$ so $k<m$).
		Since $v \in \ker D\Phi(x_0)$, 
		$v$ is of the form $v = \sum_{j=m+1}^{\card J} \lambda_i \frac{\partial}{\partial z_i}$,
		for some $\lambda_i \in \C$. 
		Now define $\xi:=\sum_{j=m+1}^{\card J} \lambda_i \frac{\partial}{\partial z_i}$.
		Clearly, $\xi$ is a germ of holomorphic vector field at $x_0$, and moreover 
		$\xi(x_0)=v$. Now let us prove that it is locally uniformly bounded in the 
		Kobayashi pseudo-metric of $X- \bigcup_{j \in J} L_j$.
		
		Let $V$ be a small neighborhood of $x_0$ in which $\xi$ is well-defined. Let $U$ be 
		a relatively compact open subset of $V$.
		There exists $r>0$ such that the complex flow  $(\phi_t)_t$ of $\xi$ is well-defined for initial
		conditions  $x \in U$ 
		for times $|t|<r$. It follows from the construction of $\xi$ that 
		for all $x \in V$, $D\Phi(\xi)=0$. Therefore, if $x \in U - \bigcup_{i \in I} L_i$, then 
		for all $|t|<r$, $\phi(t,x) \notin \bigcup_{i \in I} L_i$.
		Moreover, from the fact that  
		$\frac{d}{dt}_{|t=0} \phi(t,x) = \xi(x)$ and that $t \mapsto \phi(t,x)$ is holomorphic and 
		defined on the 
		disk of radius $r$, the definition of the Kobayashi pseudo-metric implies that 
		$$\rho_K(x;\xi(x)) \leq \frac{1}{r}.$$
	\end{proof}

	\begin{lem}\label{lem:pairwisetransverse}
		Let $X$ be a complex manifold, and $(L_i)_{i \in I}$ be a family of weakly hypersurfaces.
		Let $J \subset I$, and $i_1, i_2 \in I \backslash J$ with $i_1 \neq i_2$. Let 
		\begin{center}
			$L:=\bigcap_{j \in J} L_j, \quad
			M_1:= \bigcap_{j \in J \cup \{i_1\}} L_j, \quad 
			M_2:= \bigcap_{j \in J \cup \{i_2\}} L_j.$
		\end{center}
		Then if $M_1 \neq M_2$ and $p \in M_1 \cap M_2$, we have
		$T_p L \subset T_p M_1 + T_p M_2$.
	\end{lem}

	\begin{proof}
		Suppose $p \in M_1 \cap M_2$ and we do not have 
		$T_p L \subset T_p M_1 + T_p M_2$. We can assume that both of the $M_i$ have codimension 1 in $L$, for otherwise
		one of them must be equal to $L$ and we are done.  This means 
		that $T_p M_1 = T_p M_2$. Let $U$ be a small neighborhood of $p$, and for any $i \in I$, let 
		$P_i : U \to \C$ be a holomorphic function such that $L_i \cap U = \{P_i=0\}$.
		Let 
		\begin{align*}
		\phi:\, &U \to \C^{J \cup \{i_1,i_2\}}\\
		&x \mapsto (P_j(x))_{j \in J \cup \{i_1,i_2\}}
		\end{align*}
		
		Let $m$ be the codimension of $L$ in $\pk$. Since $T_p M_1=T_p M_2$, the rank of $D \phi(p)$ is $m+1$. By definition of weak 
		transversality, up to reducing $U$, the rank of $D\phi$ is $m+1$ on $U$. But this implies that 
		$\codim M_1 \cap M_2 = m+1$, which is impossible since $M_1$ and $M_2$ both have 
		(pure) codimension $m+1$ and $M_1 \neq M_2$.
	\end{proof}

	\section{Post-critical finiteness all the way down}

	Recall the notion of (analytic) branched cover:
	
	\begin{defi}
		Let $X, Y$ be analytic sets, and $f: X \to Y$ an analytic map. We say that 
		$f$ is a branched cover if $f$ is open, surjective and proper.
	\end{defi} 
	 
	It is then known that there is a unique closed minimal hypersurface $D$ in $Y$ 
	(possibly empty) such that $f: X - f^{-1}(D) \to Y-D$ is a covering map.
	We call $D$ the critical value set, denoted by $\cv(f)$. If $D$ is empty, 
	then we say that $f$ is unbranched.
	
	If $X$ and 
	$Y$ are smooth 
	complex manifolds, the set 
	$$\crit(f):=\{x \in X, Df(x) \text{ is not invertible}\}$$
	 is either
	empty or a closed analytic hypersurface of $X$. We call it the critical set of $f$, and 
	its image is exactly the critical value set of $f$.
	 
	\begin{defi}
		Let $X$ be an analytic set, and let $f: X \to X$ be a branched cover. 
		Let $\pf(f,X)$ denote $\overline{\bigcup_{n \geq 0} f^n(\cv(f))}$: $\pf(f)$ is called the
		post-critical set of $f$. We say that $f$ is post-critically finite (PCF) if $\pf(f,X)$ is non-empty and
		analytic with only finitely  many irreducible components.
	\end{defi} 
	
	If $f$ is  unbranched, then $\pf(f,X)=\emptyset$.
	When there is no ambiguity, we will just write $\pf(f)$.
	
	Note that any endomorphism of $\pk$ is a branched cover, see \cite{ueda1998critical}.
	Let us now define the notions of post-critical finiteness all the way down:

	\begin{defi}\label{def:pcfatwd}
		Let $X$ be an irreducible analytic set, and 
		$f: X \to X$ be a post-critically finite endomorphism. Let $f_0:=f$, $\Omega_0=X$, and denote by
		$P_0$ the post-critical set of $f$. Let us make the following definition inductively on $m$:
		if the restriction 
		of $f_m$ to each irreducible component of $\Omega_m$ is either unbranched or PCF, then
		\begin{itemize}
			\item $\Omega_{m+1}$ is the union of the $f_m$-periodic irreducible components of $P_{m}$,
			and $k_m$ is the least common multiple of the periods
			\item $f_{m+1}$ is the restriction of $f_m^{k_m}$ to $\Omega_{m+1}$
			\item $P_{m+1}$ is the union of the post-critical sets of $f_{m+1}$ restricted 
			to each irreducible component of $\Omega_{m+1}$
		\end{itemize}
	\end{defi}
	
	Note that if the restriction of $f_m$ to each irreducible component of $\Omega_m$ is PCF or unbranched,
	then $\Omega_{m+1}$ is either empty or an analytic set of pure codimension
	$\dim \Omega_m - 1$.

	\begin{defi}
		Using the same notations as the definition above:
		\begin{itemize}
			\item If for some $k \in \N$ the restriction of $f_m$ to each irreducible component of $\Omega_m$ is PCF  for every $m \leq k$ , then we say that $f$ is PCF of order $k+1$.
			\item If $f$ is PCF of order $\dim X$, we say that $f$ is PCF all the way down.
			\item If for all $k \leq \dim X-1$  the restriction 
			of $f_m$ to each irreducible component of $\Omega_m$ is either unbranched or PCF, then 
			we say that $f$ is \emph{weakly PCF all the way down}.
		\end{itemize}
	\end{defi}

	Note that if $f$ is PCF all the way down then it is weakly PCF all the way down, 
	and that if $f$ is weakly PCF all the way down then for all $m \leq \dim X$, $\Omega_m$ is either empty or 
	an
	analytic subset of $X$ of pure codimension $m$.
	
	\begin{theo}\label{th:pcfatwdcover}
		Let $X$ be a complex manifold and let $f: X \to X$ be a PCF branched covering. Let 
		$(L_i)_{i \in I}$ be the periodic irreducible components of $\pf(f,X)$ and assume that
		for any $J \subset I$, $\bigcap_{j \in J} L_j$ is smooth. Then:
		\begin{enumerate}
			\item $f$ is weakly PCF all the way down and the irreducible components of $\Omega_m$ are of the form 
			$L_{i_1} \cap \ldots \cap L_{i_m}$ for some $(i_1, \ldots , i_m) \in I^m$.
			\item In fact, for any $(i_1, \ldots, i_m) \in I_m$, if we set $L:= L_{i_1} \cap \ldots \cap L_{i_m}$,
			then 
			$f_m: L \to L$
			is either unbranched or PCF.
		\end{enumerate}
	\end{theo}
	
	Before we prove Theorem \ref{th:pcfatwdcover}, we will need the following lemma:
	
	\begin{lem}\label{lem:critintersect}
		Let $X$ be a complex manifold, and let
		$f: X \to X$ be a branched cover. Let $C$ be a smooth irreducible 
		hypersurface of $X$, and assume that $f(C)=C$. Then the critical set of the 
		restriction map $f_{|C}: C \to C$ is included in $\bigcup_i  C \cap C_i$,
		where the $C_i$ are the irreducible components of the critical set of $f$ other than $C$.
	\end{lem}
	
	\begin{proof}
		This is clear if $C$ is not a component of the critical set of $f$, since 
		any critical point for the restriction of $f$ to $C$ is also a critical 
		point for the ambient map $f: X \to X$. Therefore, let us assume that $C \subset \crit(f)$.
		Let $x$ be a critical point of the restriction of $f$ to $C$.

		All our arguments in this lemma are local.
		Therefore if $\dim X=k$, we may assume without loss of generality that 
		$C$ is locally written $\{ z_k=0\}$ in local holomorphic coordinates 
		$z=(z_1, \ldots , z_k)$ near $x$ and $f(x)$, and that $x=(0,\ldots, 0) \in \C^k$.
		In those coordinates, we may write
		$f(z_1, \ldots , z_k)=(f_1(z), \ldots, f_k(z))$: then $f_k(z_1, \ldots ,z_{k-1},0)=0$
		for every $(z_1, \ldots ,z_{k-1}) \in \Delta^k$. Therefore, for every $i \leq k-1$, 
		$\partial_i f_k =0$ on $C$. So locally near $x$, for every $z \in C$, $Df(z)$ can be represented
		 by the matrix:
		\[ Df(z)=
		\left(
		\begin{array}{c|c}
		A_k & * \\ \hline
		0 & \partial_k f_k 
		\end{array}
		\right)
		\]
		and $A_k$ is $Df_{|C}(z)$. Moreover, since $C$ is in the critical set of $f$, we have
		$$\jac f = \det A_k \times \partial_k f_k=0$$ on $C$. Also, $\det A_k$ cannot vanish 
		identically on $C$ because $f_{|C}: C \to C$ is finite to one
		(since the ambient map $f: X \to X$ a branched cover). So $\partial_k f_k$ vanishes 
		identically on $C$.
		Therefore, there is some maximal $m \in \N^*$ and some germ of holomorphic function
		$g$ such that
		$f_k(z_1, \ldots, z_k)=z_k^m g(z_1, \ldots, z_k)$, and $z_k$ does not divide $g$. Then note that 
		$z_k^m$ divides $\partial_i f_k$ for all $i<k$, but does not divide $\partial_k f_k$, while
		$z_k^{m-1}$ does divide $\partial_k f_k$
		(indeed, for all $i<k$, $\partial_i f_k=z_k^m \partial_i g$ and 
		$\partial_k f_k=m z_k^{m-1} g + z_k^{m} \partial_k g$).
		
		For every $i \leq k$, let $A_i$ be the minor of size $k-1$ corresponding to
		the $(k,i)$ entry of the matrix representation of $Df$ in the local coordinates 
		$(z_1, \ldots, z_k)$. 
		Then we have, by expanding the last row:
		$$\jac f = \sum_{i \leq k} \partial_i f_k \det A_i.$$
		
		Now if we let
		$$r= mg \det A_k + z_k\sum_{i \leq k} \partial_i g \det A_i,$$
		notice that $\jac f(z)=z_k^{m-1} r(z)$. Moreover, since $\det A_k$ is not identically zero on $C$
		and neither is $g$, the function $r$ is not identically zero on $C$.
		Now remember that by assumption, $x=0$ is in the critical locus of $f_{|C}$: therefore
		$\det A_k(0, \ldots, 0)=0$. This means that $r(0,\ldots,0)=0$, that is to say,
		$x$ belong to the analytic set $r^{-1}(0)$. Since $r$ is not identically zero on 
		$C$, this means that $x$ belong to some irreducible component of the zero locus of 
		$\jac f$ that is not $C$, which is the desired result.
	\end{proof}
	
	We will also need the following lemma due to Ueda:
	
	\begin{lem}\label{lem:ueda}(\cite{ueda1998critical}, Lemma 3.5)
		Let $X, Y$ be complex manifolds, and $f: X \to Y$ a holomorphic branched cover
		with branch locus $D \subset Y$. If $x \in \sing(f^{-1}(D))$, then $f(x) \in \sing(D)$. 
	\end{lem}
	
	Now we can prove Theorem \ref{th:pcfatwdcover}:
	
	\begin{proof}[Proof of Theorem \ref{th:pcfatwdcover}]
		Let us prove by induction on $0 \leq m \leq \dim X-1$ that 
		for any $(i_1, \ldots, i_m) \in I^m$,
		the restriction of $f_m$ to $L:=L_{i_1} \cap \ldots \cap L_{i_m}$ (with the convention
		that $L=X$ if $m=0$) is either unbranched 
		or PCF and that the irreducible components of $\pf(f_m,L)$ are of the form 
		$L_{i_1} \cap \ldots \cap L_{i_m} \cap L_{i_{m+1}}$.

		If $m=0$, then by definition the irreducible components of $\Omega_1$ are some of the $L_i$, and 
		by assumption $f: X \to X$ is PCF, so there is nothing to prove.
		
		Now let $0 \leq m \leq \dim X-1$ be such that 
		for any $(i_1, \ldots, i_m) \in I^m$,
		the restriction of $f_m$ to $L:=L_{i_1} \cap \ldots \cap L_{i_m}$ (with the convention
		that $L=X$ if $m=0$) is either unbranched 
		or PCF and that the irreducible components of $\Omega_{m+1}$ are of the form 
		$L_{i_1} \cap \ldots \cap L_{i_{m+1}}$.

		Let $N:=L_{i_1} \cap \ldots \cap L_{i_m}$ of codimension $m$ (if it has a codimension $n<m$, then it can be written
		as $L_{j_1} \cap \ldots L_{j_n}$ and the desired properties follow from the induction hypothesis).
		By definition of $f_{m+1}=f_m^{\circ k_{m+1}}$, the map $f_{m+1}$ fixes all of its post-critical components, so that 
		$\pf(f_{m+1},N) = \cv(f_{m+1},N)$ and $f_{m+1}$ restricts to an endomorphism of 
		any periodic component of $\pf(f_{m}, N)$.
		Let $L=L_{i_1} \cap \ldots L_{i_{m+1}}$ be a periodic component of $\pf(f_m, N)$, and let
		$C$ be an irreducible component of $\crit(f_{m+1},L)$ (if $\crit(f_{m+1},L)=\emptyset$, then $f_{m+1}: L \to L$ is unbranched and there is nothing to prove).
		
 		By Lemma \ref{lem:critintersect}, there is a component $C'$ of $\crit(f_{m+1},N)$ with $C' \neq L$ such that 
 		$C \subset C' \cap L$. Therefore, $C \subset \sing(f_{m+1}^{-1}(\cv(f_{m+1},N)))$. 
 		In fact, for any $n \in \N$, $C \subset \sing(f_{m+1}^{-n}(\cv(f_{m+1},N)))$ . By Lemma \ref{lem:ueda} 
 		applied to $f_{m+1}^n$ for any arbitrary
 		$n \in \N$,
 		we have $$f_{m+1}^n(C) \subset \sing(\cv(f_{m+1}^n,N)) \subset \sing(\pf(f_m,N)).$$
 		Hence the increasing sequence of analytic sets of pure codimension $m+2$ given by 
 		$$C_N:=\bigcup_{n \leq N} f_1^n(C)$$
 		is contained in the analytic set of pure codimension $m+2$ given by $\sing(\pf(f_m,N))$. Moreover, 
 		by assumption, each intersection of components of $\pf(f,X)$ is smooth
 		so the induction hypothesis implies that the singular part of $\pf(f_m,N)$ consists 
 		in points belonging to at least two different irreducible components of $\pf(f_m, N)$. Those are 
 		of the form $L_{i_1} \cap L_{i_{m+2}}$, $(i_1,i_{m+2}) \in I$.
 		In particular, $\sing(\pf(f_m,N))$  has only finitely many irreducible components, so the sequence $C_N$ is eventually
 		stationnary. Since this holds for any irreducible component $C$ of $\crit(f_{m+1},L)$, we have proved that 
 		$f_{m+1}: L \to L$ is PCF and that the irreducible components of $\pf(f_{m+1}, L)$ are of the form 
 		$L_{i_1} \cap L_{i_{m+2}}$.
		This finishes the proof.

	\end{proof}

	An endomorphism $f: \pk \to \pk$ may be represented in homogeneous coordinates by $k+1$ homogeneous polynomials
	of a same degree $d \geq 1$. The integer $d$ is called the algebraic degree of $f$ (not to be confused with the
	topological degree, which is $d^k$).
	The following classical lemma is due to Dinh and Sibony:
	
	\begin{lem}[\cite{dinh2010dynamics}, Lemma 1.48]\label{lem:topdegr}
		Let $f: \pk \to \pk$ be an endomorphism of algebraic degree $d$, and let $L \subset \pk$ be an algebraic set 
		of pure dimension $m$, such that $L=f(L)$. Then for $n$ large enough, $f^n$ fixes all irreducible components of $L$, and 
		the restriction of $f^n$ to $L$
		has topological degree $d^{nm}$.
	\end{lem}

	\begin{coro}
		Let $f: \pk \to \pk$ be a PCF endomorphism such that $\pf(f,\pk)$  is a union
		of hyperplanes. Then $f$ is PCF all the way down.
	\end{coro}

	\begin{proof}
		In this case we do not need to use Lemma \ref{lem:topdegr}.
		Any intersection of hyperplanes is a projective subspace, hence smooth. 
		In view of Theorem \ref{th:pcfatwdcover}, it is enough to prove that if $L \subset \pk$ 
		is a projective subspace with $f^p(L)=L$ for some $p \in \N^*$, then the restriction
		map $f^p: L \to L$ cannot be unbranched. 
		 Let $d$ be the algebraic degree of $f$ on $\pk$: $d \geq 2$
		since $f$ is PCF and therefore has non-empty critical set. After a change of coordinates,
		we may assume that $L$ can be written as 
		$\{z_k = \ldots = z_m=0\}$ in homogeneous coordinates.
		In those homogeneous coordinates, $f$ is of the form
		$$(z_0: \ldots :z_{k}) \mapsto (P_0(z_0, \ldots, z_{k}), \ldots, P_{k}(z_{0}, \ldots, z_{k})$$
		where the $P_i$ are homogeneous polynomials of degree $d^{p}$.
		 The restriction of $f^p$ 
		to $L$ is then of the form 
		$$(z_0: \ldots :z_{m-1}) \mapsto (P_0(z_0, \ldots, z_{m_1}, 0, \ldots,0), \ldots, P_{m-1}(z_{0}, \ldots, z_{m-1}, 0, \ldots,0))$$
		and since $f^p$ has no indeterminacy points, $f^p: L \to L$ has algebraic degree $d^p$ so
		its Jacobian (in homogeneous coordinates) is a non-constant polynomial, so the critical locus
		of $f^p: L \to L$ is not empty.
	\end{proof}

	\section{Absence of rotation domains}
	
	Let $Z$ be an analytic set, and $f: \pk \to \pk$ be a holomorphic endomorphism.
	Let $h: Z \to \pk$ be an analytic map.
	Following Ueda (\cite{ueda1998critical}), we say that a family of maps $g_n: Z \to \pk$ is a 
	family of lifts of $h$ by iterates of $f$ if for all $n \in \N$, there is $m_n \in \N$ such that 
	$f^{m_n} \circ g_n=h$.

	We will rely quite heavily on the following useful theorem, due to Ueda (\cite{ueda1998critical}, Th. 2.1):
	
	\begin{theo}\label{th:normalitylift}
		Let $Z$ be an analytic set, and $f: \pk \to \pk$ be a holomorphic endomorphism.
		Let $h: Z \to \pk$ be an analytic map. Let $(g_n)_{n \in \N}$ be a family of lifts of $h$ by iterates
		of $f$. Then $(g_n)_{n \in \N}$ is a normal family.
	\end{theo}

	The following result is a straightforward adaptation of Theorem 4.15 of \cite{ueda1998critical}.
	We include a proof for the convenience of the reader.
	
	\begin{theo}\label{th:rotdombdry}
		Let $f: \pk \to \pk$ be an endomorphism of algebraic degree $d \geq 2$.
		Let $L$ be a closed regular subvariety of $\pk$ with $L=f(L)$, and assume that the restriction
		$f: L \to L$ has a rotation domain $V$. Then $\partial V$ (the boundary of $V$ in $L$) is included
		in the post-critical set of the restriction map $f: L \to L$.
	\end{theo}

	\begin{rem}
		The assumption that $f: L \to L$ extends to an endomorphism of $\pk$ is necessary:
		indeed, there are automorphisms of projective varieties (which 
		therefore have empty post-critical set) with rotation domains that are not 
		all of $L$. The crucial point here is the normality of the inverse branches
		afforded by Theorem \ref{th:normalitylift}.
	\end{rem}
	
	\begin{proof}[Proof of Theorem \ref{th:rotdombdry}]
		Let $x \in \partial V$. Assume that $x$ does not belong to the post-critical set $\pf(f,L)$
		of $f: L \to L$. Let $D$ be a simply connected neighborhood of $x$ in $L$ that does not 
		meet the post-critical set of $f: L \to L$.
		Let $(n_j)_{j \in \N}$ be a subsequence such that $f^{n_j}$ converges locally uniformly 
		to the identity on $S$. For all $j \in \N$, $f^{n_j}: V \to V$ is an automorphism of $S$, 
		so we can consider $h_j:=(f^{n_j})^{-1}: V \to V$. Moreover, since $D$ is simply connected 
		and disjoint from the post-critical set of $f: L \to L$, for all $j \in \N$ 
		branches of $f^{-n_j}$ are well-defined, so we can extend $h_j$ to $D$.
		
		By Theorem \ref{th:normalitylift}, the family $(h_j)_{j \in \N}$
		is normal on $D$, and from the fact that $V$ is a rotation domain,
		 it is easy to see that $(h_j)_{j \in \N}$ converges to the identity on $V$, 
		 and therefore also on $D$ by the identity principle. 
		Let $W$ be a connected, relatively compact open subset of $D$ containing $x$ such that
		for all $j \in \N$ large enough, $W \subset h_j(D)$. Then  $f^{n_j}_{|W}$ converges
		to $\id_{|W}$, hence $W$ is included in the rotation domain $V$. This contradicts 
		the fact that $x \in \partial V$.
	
	\end{proof}

	\begin{coro}\label{coro:pcfnorot}
		Assume that $f: \pk \to \pk$ is PCF, and call $L_i, i \in I$ the irreducible components of 
		$\Omega_1$. Let $J \subset I$, $L:=\bigcap_{j \in J} L_j$ and 
		$m$ the least common multiple of the periods of the $L_j, j \in J$. Then $f^{m}:L \to L$ 
		has no rotation domain.
	\end{coro}

	\begin{proof}
		Assume that $f: L \to L$ does have a rotation domain $V$.
		By Theorem \ref{th:rotdombdry}, the boundary $\partial V$ is contained in 
		the post-critical set of $f: L \to L$. By Theorem \ref{th:pcfatwdcover}, 
		$f: L \to L$ is either unbranched or PCF.

		Let us first assume that $f: L \to L$ is 
		unbranched: then $\partial V=\emptyset$, so $L=V$.  In particular, $f: L \to L$ is an automorphism
		of $L$, which contradicts Lemma \ref{lem:topdegr} (indeed, $f$ must have algebraic degree at least $2$
		since it is PCF).
		
		Let us now assume that $f: L \to L$ is PCF instead of unbranched. Then $\partial V$
		is contained in an algebraic hypersurface of $L$. Since hypersurfaces of $L$ do not disconnect $L$, $L\backslash\partial V$ is connected,
		so $V=L\backslash\partial V$ and $V$ contains the complement of the post-critical set of $f: L \to L$.
		By Lemma 4.13 of \cite{ueda1998critical}, this means that $f: L \to L$ is in fact an automorphism of 
		$L$ and that $V=L$. Again, this contradicts Lemma \ref{lem:topdegr}.
	\end{proof}

	\section{Eigenvalues of periodic points of PCF endomorphisms}
	
	Recall the following definition from topological dynamics:
	
	\begin{defi}\label{def:lyapunstable}
		Let $X$ be a Hausdorff topological space, $f: X \to X$ a continuous map, and $p \in X$ a fixed point.
		The fixed point $p$ is said to be Lyapunov stable if for every neighborhood $V$ of $p$, there exists 
		an open set $U$ containing $p$ such that for all $n \in \N$, $f^n(U) \subset V$.
		We will also say that $p$ is Lyapunov unstable if $f$ has a local inverse $g$ fixing $p$ such that 
		$p$ is Lyapunov stable for $g$.
	\end{defi}

	\begin{rem}
		In the case where $X$ is a complex manifold and $f$ is holomorphic, a necessary (but not sufficient) 
		condition
		for Lyapunov stability is for the eigenvalues of $Df(p)$ to have modulus smaller or equal to one. 
		A sufficient (but not necessary) condition is that every eigenvalue of $Df(p)$ has modulus strictly
		smaller than one.
	\end{rem}

	\begin{lem}\label{lem:brinverse}
		Let $L$ be a complex manifold, and $f: L \to L$ a PCF ramified cover. Let 
		$p \in \reg(\pf(f,L)) \backslash \crit(f,L)$ be a Lyapunov unstable fixed point for $f: \pf(f,L) \to \pf(f,L)$.
		Then there is a simply connected open subset $U \subset L$ containing $p$ such that
		for all $n \in \N$, there is a well-defined holomorphic lift $g_n: U \to L$ of $f^n$ fixing $p$.
	\end{lem}

	\begin{proof}
		Let us begin with the case $n=1$. Let $V$ be a small enough neighborhood of $p$ in $L$ 
		that $f$ admits a unique local inverse $g_1$ defined on $V$, with $g_1(p)=p$.
		Let $U$ be an open subset given by the Lyapunov stability of $p$ as a fixed point of $g_1$.
		Up to restricting $U$, we can take $U$ to be simply connected. 
		
		We will prove the following statement by induction: for all $n \in \N^*$, there is
		a well-defined holomorphic lift $g_n: U \to L$ of $f^n$ fixing $p$, such that 
		$$g_n(U) \cap \pf(f,L) \subset g_1^n(U) \cap \pf(f,L).$$
		
		Now assume that $g_n$ is constructed, and let us set $V_n:=g_n(U)$.
		Let us prove that $g_1$ admits analytic continuations along any path in $V_n$ starting at $p$.
		Let $\gamma: [0,1] \to V_n$ be a path with $\gamma(0)=p$. By the induction hypothesis and the fact
		that $p$ is a Lyapunov unstable point,
		$$V_n \cap \pf(f,L) \subset g_1^n(U) \cap \pf(f,L) \subset V \cap \pf(f,L).$$ 
		Using the compacity of $\gamma([0,1])$, we may find  a finite
		subcover of $\gamma([0,1])$ by small balls $(U_i)_{i \leq m}$ such that either $U_i \subset V$
		or $U_i \cap \pf(f,L)=\emptyset$, and then define an analytic continuation of $g_1$ along $\gamma$
		in the usual way, using compatible local inverses on each $U_i$.
		Since $U$ is simply connected and $g_n$ is univalent, $V_n$ is also simply connected, so by
		the Monodromy Theorem $g_1$ extends to a single-valued map $g_{1,n}$ defined on $V_n$
		\footnote{We use the notation $g_{1,n}$ to emphasize the fact that the values of this extension depend 
			a priori on $n$ on $V_n \backslash V$.}. We then set:
		$$g_{n+1}:=g_{1,n} \circ g_n \text{ on }U.$$
		Clearly this defines a lift of $f^n$ fixing $p$.
		
		Now set $V_{n+1}:=g_{n+1}(U)$. To finish the proof by induction, we need to prove that
		$V_{n+1} \cap \pf(f,L) \subset g_1^{n+1}(U) \cap \pf(f,L)$. Let $x \in V_{n+1} \cap \pf(f,L)$, and let 
		$y$ be the unique element of $V_n$ such that $x=g_{1,n}(y)$. Then $f(x)=y \in \pf(f,L)$ by forward invariance
		of $\pf(f,L)$, 
		so by the induction hypothesis, $y \in g_1^n(U) \cap \pf(f,L)$. So finally $x=g_{1,n}(y)=g_1(y) \in g_1^{n+1}(U) \cap \pf(f,L)$
		(the second equality comes from the fact that $g_{1,n}=g_1$ on $V$),
		and the proof is finished.
	\end{proof}

	\begin{lem}\label{lem:regandcrit}
		Let $f: \pk \to \pk$ be an endomorphism, and assume that $L$ is a subvariety such that 
		$L=f(L)$ and $f: L \to L$ is PCF. Assume that $p=f(p) \in \reg(\pf(f,L)) \cap \crit(f,L)$.
		Up to replacing $f$ by an iterate, there is a fixed component 
		$P_i$ of $\pf(f,L)$ such that the image of $Df(p)_{|T_p L}$
		is included in $T_p P_i$.
	\end{lem}

	\begin{proof}
		Up to replacing $f$ by an iterate, we may assume that all irreducible
		components of $\crit(f,L)$ are mapped by $f$ to fixed 
		post-critical components.
		According to \cite{grauert1958komplexe}, Satz 10, we can find local coordinates
		(different in domain and range) in which $f: L \to L$ takes locally near $p$ (in $L$) 
		the form $(z_1, z_2, \ldots) \mapsto (z_1^m, z_2, \ldots)$.
		In particular, $p$ belongs to the regular part of $\crit(f,L)$, and
		$Df(p)$ maps $T_p \crit(f,L)$ injectively into $T_p \pf(f,L)$.
		Since $Df(p)$ has rank at most $\dim L-1$ at $p$, the image of $Df(p)$ must be exactly $T_p \pf(f,L)$.
	\end{proof}
	
	%
	%
	%
	%
	
	The following lemma builds on ideas from \cite{fornaess1994complex}:

	\begin{lem}\label{lem:notinpf}
		Let $f: \pk \to \pk$ be an endomorphism, and assume that $L$ is a subvariety such that 
		$L=f(L)$ and $f: L \to L$ is weakly PCF. Let $p=f(p) \notin  \crit(f,L)$, and assume that
		there is an open neighborhood $U$ of $p$ in $L$ such that
		for all $n \in \N$ there is a holomorphic lift $g_n: U \to L$ of $f^n$ with
		$g_n(p)=p$. Then $p$ is Lyapunov unstable, and at least one of its eigenvalues is repelling.
		Moreover, $p$ is linearizable through an analytic submanifold tangent to the subspace associated to 
		neutral eigenvalues.
	\end{lem}
	
	\begin{proof}
		By Theorem \ref{th:normalitylift}, the family $\{g_n, n \in \N\}$ is normal. 
		Moreover, $Dg_n(p)=Df(p)^{-n}$, so $Df(p)$ cannot have any attracting eigenvalue, 
		for otherwise $Dg_n(p)$ would blow up, contradicting the normality of the $g_n$, $n \in \N$.
		
		Therefore all eigenvalues must be either neutral or repelling; assume for a contradiction
		that none of them are repelling. Then every eigenvalue of $Df(p)$, hence of $Dg_1(p)$, have modulus one.
		Moreover, the sequence $Dg_n(p)=Dg_1^n(p)$ is bounded; by considering a Jordan decomposition, one 
		easily sees that this implies that $Dg_1(p)$ 
		is unitary. By Lemma 6.8 of \cite{fornaess1994complex},
		$g_1$ and therefore $f$ are linearizable in an open neighborhood of $p$ in $L$. This implies that $p$
		belongs to a rotation domain of $f: L \to L$, which contradicts Corollary \ref{coro:pcfnorot}.
		
		Let us prove that $p$ is Lyapunov unstable: let 
		$V$ be an open neighborhood of $p$. Up to reducing $V$, we may assume that each $g_n$ coïncide with 
		$g_1^n$ on $V$. Since $\{g_1^n, n \in \N\}$ is normal on $V$, there is a neighborhood $U \subset V$ 
		such that for all $n \in \N$, $g_1^n(U) \subset V$. Thus $p$ is Lyapunov unstable.
		
		Finally, let us prove the linearizing statement. Let
		$$T_p L = E^c \oplus E^u$$
		be the decomposition of $T_p L$ into invariant subspaces associated to neutral and repelling eigenvalues
		respectively. We will prove that there is a germ of invariant analytic submanifold
		$W_{\mathrm{loc}}^c(p)$ tangent to $E^c$ on which $f$ is linearizable. Let 
		$h=\lim_{j \to \infty} g_{n_j}$
		be a limit function. Up to extracting a further subsequence, we may assume that $n_{j+1}-n_j \to \infty$,
		and that $g_{n_{j+1}-n_j}$ converges to some other limit function $k: U \to L$. Observe that we have
		$$k \circ h = h.$$
		Let $S:=\{z \in U,\, k(z)=z\}$. This is an analytic subset of $U$, and since $Dk(p)$ is the projection
		on $E^c$ with kernel $E^u$, we have $\dim S \leq \dim E^c$. Moreover, the rank of $Dh(p)$ is 
		equal to $\dim E^c$, and $h(U) \subset S$, so in fact the local dimension of $S$ at $p$ is exactly 
		$\dim E^c$ and $h(U)=S$ locally near $p$. It is easy to check that $S$ is invariant under $g_1$.
		Again, using the fact that the iterates of $Dg_1(p)$ are bounded and looking at its Jordan decomposition, 
		we can see that all the Jordan blocks associated to neutral eigenvalues must have size at most 1, and so 
		$Dg_1(p)$ restricted to $T_p S$ is unitary. Again, by Lemma 6.8 of \cite{fornaess1994complex}, this implies
		that $g_1$ is linearizable on $W_{\mathrm{loc}}^c(p):=S$, and therefore the same holds for $f$.
	\end{proof}

	Let us now prove Theorem B on the eigenvalues of cycles of PCF endomorphisms of $\pk$:

	\begin{theoB}
		Let $f: \pk \to \pk$ be a PCF endomorphism such that the irreducible components of
		$\pf(f,\pk)$ are weakly transverse.
		Let $p$ be a periodic point of period $l$ for $f$. Then 
		\begin{enumerate}
			\item $p$ has no parabolic or non-zero attracting eigenvalue 
			\item either $p$ has at least one repelling eigenvalue, or else all eigenvalues are zero
			(ie $Df^l(p)$ is nilpotent)
			\item if $p$ doesn't have zero as an eigenvalue, then $p$ is Lyapunov unstable
			and at least one of its eigenvalues is repelling.
		\end{enumerate}
	\end{theoB}

	\begin{proof}
		Up to replacing $f$ by an iterate, we lose no generality in considering only fixed points
		and in assuming that all irreducible components of the $\Omega_m$, $0 \leq m \leq k$, are fixed.
		Let $p=f(p)$ be a fixed point of $f$.
		
		If $p \notin \pf(f,\pk)$, 
		we can choose a simply connected neighborhood $U$ of $p$ in $L$, and 
		for all $n \in \N$ we have a well-defined inverse branch $g_n: L \to L$ fixing $p$.
		Then by Lemma \ref{lem:notinpf}, we are done.
		Therefore we will now assume that $p \in \pf(f,\pk)$. 
		
		Let $m_0$ be the greatest $m \leq k-1$
		such that $p \in \Omega_m$. We shall prove the following by descending induction on the codimension $m$, 
		for $0 \leq m \leq m_0$: \\
		
		For any irreducible component $L$ of $\Omega_m$ containing $p$:
		\begin{enumerate}
			\item the non-zero eigenvalues of $p$ associated to vectors tangent to
			$L$
			are not attracting or parabolic
			\item if none of those eigenvalues are repelling, then they are all 0 
			\item if $p \notin \crit(f,L)$, then $p$ 
			is 
			a Lyapunov unstable fixed point of $f$ in $L$.
		\end{enumerate}

		If $m_0=k-1$ and $L$ is an irreducible component of $\Omega_{k-1}$ containing $p$, then 
		$f: L \to L$ is weakly PCF according to Theorem \ref{th:pcfatwdcover}, and 
		$L$ is a compact Riemann surface. Therefore $L$ must be either a torus or a projective line
		(indeed, in all other cases, $f$ would have to restrict to a degree one self-map of $L$, contradicting Lemma \ref{lem:topdegr}).
		If $L$ is a torus, then all cycles are repelling, and if $L$ is a projective line, it is classical
		that $f: L \to L$ may only have super-attracting or repelling cycles.
		
		If $m_0<k-1$ and $L$ is an irreducible component of $\Omega_{m_0}$ containing $p$,  then by Theorem \ref{th:pcfatwdcover}, $f: L \to L$ is an unbranched cover, and $\pf(f,L)=\emptyset$. In particular, 
		for any simply connected neighborhood $U$ of $p$ we can take inverse branches $g_n$ of $f^n$ fixing $p$.
		Then,by Lemma \ref{lem:notinpf} $p$ is Lyapunov unstable, has at least one repelling 
		eigenvalue and no attracting or parabolic eigenvalue. Thus the statement is proved
		for $m=m_0$.
		
		Now let $m \geq 1$ and 
		assume that $p \in \Omega_m$ and that the induction hypothesis is satisfied for $m$.
		Let $L$ be an irreducible component of $\Omega_{m-1}$ containing $p$.
		Note that by assumption, $p \in \pf(f,L)$ since $p \in \Omega_m$.
		We will distinguish 3 mutually exclusive cases:
		\begin{enumerate}
			\item $p \in \reg(\pf(f,L)) \cap \crit(f,L) $
			\item $p \in \reg(\pf(f,L)) \backslash \crit(f,L)$
			\item $p \in \sing(\pf(f,L))$ 
		\end{enumerate}

		Case (1) follows from Lemma \ref{lem:regandcrit} and the induction hypothesis.

		Assume we are in case (2). Let $P$ be the irreducible post-critical component of
		$\pf(f,L)$ containing $p$. By item (3) of the induction hypothesis, $p$ is a Lyapunov unstable
		fixed point of the restriction $f: P \to P$. By Lemma \ref{lem:brinverse} and 
		Lemma \ref{lem:notinpf}, $p$ is also a Lyapunov unstable fixed point for $f: L \to L$,
		and it has at least one repelling eigenvalue, so the three items of the induction are proved.

		Finally, let us assume that we are in case (3). Since we assumed that any intersection
		of components of $\Omega_1$ is smooth, Theorem \ref{th:pcfatwdcover} implies that 
		$p$ belongs to at least two different components $P_1$ and $P_2$ of $\Omega_{m}$.
		Since the irreducible components of $\Omega_1$ are weakly transverse, Lemma \ref{lem:pairwisetransverse}
		implies that
		$T_p L \subset T_p P_1 + T_p P_2$. 
		Also note that each of the $T_p P_i$ are 
		invariant under $Df(p)$. Therefore any eigenvalue of $Df(p)$ on $T_p L$ must be an eigenvalue
		on either $T_p P_1$ or $T_p P_2$. This together with the induction hypothesis imply
		that all three desired items of the induction are true.

		Thus the induction is finished and the Theorem is proved, since it corresponds to the 
		case $m=0$.
	\end{proof}

	\section{Super-attracting basins are the only Fatou components}

	We are now able to prove Theorem C:
	
	\begin{theoC}\label{th:main}
		Let $f: \pk \to \pk$ be a PCF endomorphism, and let $L_i, i  \in I$ be the irreducible components
		of its post-critical set $\pf(f)$. Assume that
		\begin{enumerate}
			\item $(L_i)_{i \in I}$ is weakly transverse
			\item $\pk \backslash\pf(f)$ is Kobayashi hyperbolic.
		\end{enumerate}
		Then the Fatou set of $f$ is empty or a finite union of basins of super-attracting cycles,
		whose points are zero-dimensional intersections of components $L_i$.
	\end{theoC}
		
		\begin{proof}[Proof of Theorem C]
			Let us start by introducing some notations. 
			Denote by $(L_i)_{i \in I}$ the family of the irreducible components of $\pf(f)$.

			For every $y \in \pf(f)$, denote by 
			$M_y$ the intersection of all irreducible components of $\pf(f)$ containing $y$:
			$$M_y =\bigcap_{i \in I_y} L_i $$
			where $I_y$ is the set of $i \in I$ such that $y \in L_i$.

			For any $i \in I_y$, there exists $(k_i, m_i) \in \N \times \N^*$ such that 
			$f^{m_i + k_i}(L_i)=f^{m_i}(L_i)$. 
			Let $m_y:= \lcm(m_i, i \in I_y)$
			and 
			$$N_y:=\bigcap_{i \in I_y} f^{k_i}(L_i).$$
			Then we have $f^{m_y}(N_y) \subset N_y$.

			For any closed hypersurface $V$, denote by $\| \cdot \|_{\pk-V}$ the Kobayashi
			pseudo-metric of $\pk-V$.

			\begin{step}
				Let $x \in \pk-f^{-1}(\pf(f))$. Then for every $v \in T_x \pk$, we have
				$$\|Df(x)\cdot v\|_{\pk-\pf(f)} \geq \|v \|_{\pk-\pf(f)}.$$
			\end{step}
			
			\begin{proof}
				Since the restriction $f: \pk-f^{-1}(\pf(f)) \to \pk-\pf(f)$ is a covering map, we have
				$$\|Df(x)\cdot v\|_{\pk-\pf(f)} =   \|v \|_{\pk-f^{-1}(\pf(f))}.$$	
				Moreover, since there is an inclusion $i: \pk-f^{-1}(\pf(f)) \to \pk-\pf(f)$, we have
				$$\|v \|_{\pk-f^{-1}(\pf(f))} \geq \|v \|_{\pk-\pf(f)}.$$
			\end{proof}
			
			\begin{step}
				Let $z \in \fatou(f)$ be a point whose orbit
				never lands in $\pf(f)$, and let $f^{n_j}$ be a subsequence converging to some map $h$ on a neighborhood
				of $z$. Let $y:=h(z)$. Then $y \in \pf(f)$, and 
				$Dh(z): T_z \pk \to T_y M_y$ is surjective.
			\end{step}
			
			\begin{proof}[Proof]
				The fact that $y$ must be in $\pf(f)$ is well-known; we refer the reader 
				to  \cite{rong2008fatou} for instance.
				Let $(P_i)_{i \in I_y}$ denote a collection of irreducible polynomials
				whose zero loci correspond to the irreducible components $L_i$ of $\pf(f)$ containing 
				$y$, and let $v \in \bigcap_{i \in I_y} \ker DP_i(y)$.
				Let us prove that there exists $w \in T_z \pk$ such that 
				$Dh(z) \cdot w = v$. Let $\xi$ be a holomorphic vector field
				defined on a neighborhood $U$ of $y$,
				such that $\xi(y)=v$ given by Lemma \ref{lem:kobtan}.
				For any $j \in \N$ large enough to ensure that $f^{n_j}(z) \in U$,
				let $w_j = (Df^{n_j})^{-1}(f^{n_j}(z)) \cdot \xi(f^{n_j}(z))$.
				Since $f$ expands the Kobayashi metric of $\pk - \pf(f)$ (Step 1), we have:
				$$\|w_j\|_{\pk-\pf(f)} \leq \|\xi(f^{n_j}(z))\|_{\pk-\pf(f)}$$
				Moreover, by Lemma \ref{lem:kobtan}, there exists a constant 
				$C>0$ independant from $j$ such that:
				$$\|\xi(f^{n_j}(z))\|_{\pk-\pf(f)} \leq C.$$
				Therefore up to extracting one more time, the sequence $w_j$ converges to some $w \in T_z \pk$.
				By continuity of $\xi$, we have that $Dh(z)\cdot w = v$.
			\end{proof}

			\begin{step}
				Let $h$ be a Fatou limit function, $z \in \fatou(f)$ and let $y:=h(z)$. Assume that 
				$Dh(z): T_z \pk\to T_y M_y$ is surjective. Then 		
				we have that the sequence $(Df^k(y)_{|T_y M_y})_{k \in \N}$ is uniformly bounded
				(in the Fubini-Study metric of $\pk$).
			\end{step}

			\begin{proof}[Proof of Step 3]		
				Let $v \in T_y L$: we shall prove that the sequence $Df^k(y) \cdot v$ is bounded.
				According to Step 1, there is $w \in T_z \pk$ such that 
				$Dh(z) \cdot w=v$. Let $j,k \in \N^*$.
				We have:
				\begin{equation*}
				Df^{n_j+k}(z) \cdot w = Df^k(f^{n_j}(z)) \circ Df^{n_j}(z) \cdot w
				\end{equation*} 
				Since $Df^{n_j}(z)$ converges to $Dh(z)$, for all $k \in \N^*$ there exists $j_0 \in \N$ such that for all $j \geq j_0$,
				\begin{equation*}
				\left| \|Df^{n_j+k}(z) \cdot w \| - \| Df^k(y) \cdot v \| \right| \leq \frac{1}{k}.
				\end{equation*}
				Finally, since $z$ belongs to the Fatou set of $f$, the sequence 
				$\|Df^n(z)\cdot w\|$ is bounded, therefore so is the sequence 
				$\|Df^k(y)\cdot v\|$.
				
			\end{proof}

			\begin{step}\label{step:submh}
				Let $U$ be a Fatou component of $f$, and let $h: U \to \pk$ be a Fatou limit function.
				There is a unique maximal $J_h \subset I$ such that 
				$h(U) \subset \bigcap_{j \in J_h} L_j:=M_h$.
				Moreover, the set $V$ of those $z \in U$ such that $L_{h(z)}=M_h$ is open and dense in $U$, and 
				$h: U \to M_h$ is a submersion.
			\end{step}
			
			\begin{proof}
				For every component $L_i$ of $\pf(f)$, let $P_i$ denote an irreducible homogenous polynomial
				on $\C^{k+1}$ such that $M_i=\{Pi=0\}$ in homogenous coordinates. Since $h(U) \subset \pf(f)$,
				we have: 
				$$\prod_{i \in I} P_i \circ h =0.$$
				Let $\phi_i=P_i \circ h$, and let $J=\{i \in I, \phi_i  \text{ does not vanish identically on  }U \}$.
				
				Then $\bigcup_{i \in J} \phi^{-1}(0)$ is a proper analytic subset of $U$ (possibly empty), so 
				$V:= U - \bigcup_{i \in J} \phi^{-1}(0)$ is open and dense in $U$.
				
				Let $z \in V$ and $y:=h(z)$. By definition of $V$, if $y$ belongs to some component $L$ of $\pf(f)$, then $h(U) \subset L$;
				therefore $h(U) \subset M_y$. Since this is true for any $z \in V$, we have that for every $z, z' \in V$,
				$M_{h(z)}=M_{h(z'}$. We can therefore set $M_h = M_h(z)$ for any arbitrary $z \in V$. 
				Moreover, by Step 1, $Dh(z): T_z \pk \to T_y M_y$ is surjective if 
				the orbit of $z$ does not meet $\pf(f)$, or in other words, if 
				$z \in V - \bigcup_{ n \geq 0} f^{-n}(\pf(f))$. But for all $n \in \N$, 
				$f^{-n}(\pf(f))$ is a proper closed algebraic hypersurface, so by Baire category
				$V - \bigcup_{ n \geq 0} f^{-n}(\pf(f))$ is dense in $V$. This implies that for all $z \in V$,
				$Dh(z): T_z \pk \to T_y M_y$ is surjective; and since $V$ is dense in $U$, 
				we have that $Dh(z): T_z  \pk \to T_y M_h$ is surjective for all $z \in U$. 
			\end{proof}
			
			\begin{step}
				With the same notations as the previous step: 
				there exists $(k,m) \in \N \times \N^*$ such that $f^{m+k}(M_h)=f^k(M_h)$,
				and
				$f^k \circ h(U)$ is in the Fatou set of the restriction of 
				$f^m$ to $N_h:=f^k(M_h)$.			
			\end{step}
			
			\begin{proof}		
				According to Step \ref{step:submh}, for any Fatou limit function $h$, the map $h: U \to M_h$ is 
				a submersion, hence open. For all $n \in \N$, $f^n \circ h$ is also a Fatou limit function, so let us denote 
				by $M_n$ the variety $M_{f^n\circ h}$. Then for all $n \in \N$, $f^n \circ h: U \to M_n$ is 
				open, which implies that all $M_n$ have the same dimension, as the ambient map $f: \pk \to \pk$ maps analytic 
				subsets to analytic subsets of same dimension.
				Since there are only finitely many different varieties of the form $\bigcap_{j \in J} L_j$ ($L_j$ being
				the irreducible components of $\pf(f)$), there are only finitely many $M_k$.
				Thus there is $(k,m) \in \N \times \N^*$ such that $M_{k+m}=M_k$.
				
				Now from the fact that $f^m$ maps an open subset of $M_k$ (namely $f^k\circ h(U)$) to another open subset 
				of $M_{k}$ (namely $f^{m+k}\circ h(U)$), it is easy to see that $f^m(M_k) \subset M_k$.
				Since $M_k$ is a closed irreducible variety, we have in fact $f^m(M_k)=M_k$.
				Finally, note that $M_k=f^k(M_h)$, again because $f^k$ maps the open subset $h(U) \subset M_h$ to 
				the open subset $f^k \circ h(U) \subset M_k$.

				Applying Step 2 to every $z \in U$, we have that for every $y \in h(U)$,
				the derivatives $Df^n_{|T_{y} M_{h}}, n \in \N$, are uniformly bounded.
				Moreover, $h(U)$ is open by Step 4 and the submersion lemma.
				Therefore, the family of maps 
				$\{f_{|h(U)}^n: h(U) \to \pk, n \in \N\}$ is normal. 
				This implies that the family $\{f_{|f^k\circ h(U)}^{\circ mn}: f^k\circ h(U) \to N_h, n \in \N\}$ 
				is normal, which is exactly saying that $f^k\circ h(U) \subset N_h$ is in the Fatou set of 
				the restriction of $f^m$ to $N_h$.
			\end{proof}
			
			\begin{step}
				Let $U$ be a Fatou component of $f$. There exists a subsequence $(m_j)_{j \in \N}$ such 
				that $f^{m_j}$ converges on $U$ to a cycle for $f$.
			\end{step}

			\begin{proof}[Proof ]
				We will prove that there is a constant Fatou limit function on $U$.
				To do this, we will prove that for every Fatou limit function $h$ on 
				$U$ such that 
				the $M_h$ given by Step 4 has dimension $\delta \geq 1$, there is a Fatou 
				limit function $\tilde h$ on $U$ such that $\dim M_{\tilde h} \leq \delta -1$. \\
				
				According to the previous step, $U':=f^k\circ h(U)$ is  an open
				subset of  the Fatou set of $g:=f_{|N_h}^m : N_h \to N_h$.
				Let $(n_j)_{j \in \N}$ be a subsequence such that $g^{n_j}$ converges to a Fatou limit function
				$h_2$ (for $g$) on $U'$. Up to extracting further, we may assume that $n_{j+1}-n_j$ tends to infinity. 
				We are going to prove that for any $z \in U'$, the rank of $Dh_2(z)$ must be less than or equal to 
				$\delta-1=\dim N_h-1$. Suppose for a contradiction that it is not the case: then by the local inversion theorem, 
				we may find an open set $W \subset U'$ with $W \subset h_2(U')$.  
				Up to extracting further, we may assume without loss of generality that $n_{j+1}-n_j$ tends to infinity. 
				From the equality $g^{n_{j+1}-n_j} \circ g^{n_j}=g^{n_{j+1}}$, we deduce that $g^{n_{j+1}-n_j}$ converges to 
				$\id_W$ on $W$; therefore, $W$ is included in a rotation domain for $g$. But this is impossible by 
				Corollary \ref{coro:pcfnorot}.
				
				Therefore, there exists 
				a Fatou limit function  $h_2$ for $g=f_{|N_h}^m$ defined on $U'=f^k\circ h(U)$, 
				such that the rank of $Dh_2$ is strictly less than $\delta-1$ everywhere on $U'$.
				
				 Moreover, $\tilde h := h_2 \circ h$ is a Fatou limit function of $f$ (on $U$),
				 as can be seen by considering iterates of the form $f^{k+m n_j + p_i}$ with $i \in \N$ 
				 large and then suitable $j \in \N$, where $(p_i)_{i\in \N}$
				 is a subsequence such that $f^{p_i}$ tends to $h$.
				Moreover, the  differential of $\tilde h$ has a rank at most $\delta-1$ everywhere on $U$. \\

				Therefore, by descending induction on $\delta$, we can find a Fatou limit function 
				on $U$ with rank $0$, or in other words, a constant Fatou limit function $h=y$,
				with $y \in \pf(f)$.
				By Step 4, we have $M_y = M_h = \{y\}$, so $\{y\}$ is a zero-dimensional intersection
				of hyperplanes in $\pf(f)$, and therefore $y$ has a finite orbit. So up to replacing $h$
				by $f^p \circ h$ for some appropriate value of $p \in \N$ 
				(we may do so since $f^p \circ h$ is also a Fatou limit function), we may assume that 
				$y$ is periodic.

			\end{proof}
			
			\begin{step}
				For every $z$ in the Fatou set of $f$, we have $\limn Df^n(z)=0$.
			\end{step}

			\begin{proof}[Proof]
				Let $U$ be a Fatou component of $f$.
				By the previous step, there is a subsequence $f^{n_j}$
				such that $f^{n_j}$ converges on $U$ to a finite set; therefore 
				$Df^{n_j}$ converges locally uniformly to zero on $U$.

				Now observe that for any $n \in \N$, 
				$$\|Df^n\| \leq  \|Df^{n_j}\| \|Df^{n-n_j}\|$$
				where $j=\max\{i , n_i \leq n\}$, and that $\|Df^{n-n_j}\|$ is uniformly bounded
				on every compact of $\Omega$ (by normality). Here we used the Fubini-Study metric on $\pk$
				(or any other continuous Hermitian metric on $\pk$).
			\end{proof}

			\begin{step}
				For every Fatou component $U$ of $f$, there is a periodic point $y$ such that $f^n$ 
				converges to that cycle, locally uniformly on $U$.
			\end{step}	
			
			\begin{proof}
				Applying Step 2 and the previous step, we obtain that for every $z \in U$ and every
				Fatou limit function $h$,
				$M_y$ has dimension zero, with $y:=h(z)$. Thus by Step 4,
				for every Fatou limit function $h$, there is a  $z \in U$ such that if
				$y:=h(z)$, we have  that $h(U)=M_y=\{y\}$, and $y$ is  a periodic point for $f$
				in view of Step 6. Since $M_y=\{y\}$, $y$ can be written as the intersection of 
				at least $k$ hyperplanes that are irreducible components of $\pf(f)$. There are
				only finitely many such points, and they are all fixed by $f$. 
				Denote by $F$ the set of those points. Let $K \subset U$ be a compact with non-empty interior,
				and let $\epsilon>0$. There exists $n_0 \in \N$ such that for all $z \in K$, for all
				$n \geq n_0$, $f^n(z) \in \bigcup_{p \in F} D(p,\epsilon)$. If $\epsilon$ is small enough,
				the disks $D(p,\epsilon)$, $p \in F$, are pairwise disjoint and 
				for all $p \neq q$, $f(D(p, \epsilon)) \cap D(q,\epsilon)=\emptyset$ (because all 
				the $p \in F$ are fixed points). Therefore, if $\epsilon$ is small enough, and 
				$n \geq n_0$, there is a unique $p \in F$ such that $f^n(K) \subset D(p,\epsilon)$, and 
				moreover for all $q \neq p$, we have $f^{n+1}(K) \cap D(q,\epsilon)=0$.
				Therefore, by induction, we must have $f^n(K) \subset D(p,\epsilon)$ for all $n \geq n_0$.
				This proves that the sequence of iterates $f_{|K}^n$ restricted to $K$ converge to $p$.
			\end{proof}

			\begin{step}
				The periodic cycles from the previous step must be superattracting.
			\end{step}
			
			\begin{proof}
				Let $p \in \pk$ be such a periodic point. Replacing $f$ by one of its iterates if necessary,
				we assume that $p$ is a fixed point. Since there is an open set $U \subset \pk$ of points 
				converging locally uniformly to $p$ under iteration, $Df(f^n(z))$ converges locally uniformly
				to $Df(p)$, and $Df^n(z)$ converges to $0$. Therefore $p$ cannot have any repelling eigenvalue. 
				By Theorem B,
				this means that $Df(p)$ is nilpotent, hence that $p$ is superattracting.
			\end{proof}

			To sum things up, we have proved that for every Fatou component $U$, the iterates
			$f^n$ converge locally uniformly on $U$ to some super-attracting periodic point $y$, that 
			is a maximal codimension intersection of irreducible components of $\pf(f, \pk)$.
			This proves Theorem C.
		\end{proof}

	\section{The case of moduli space maps}\label{sec:koch}

	In this final section, we recall Koch's results from \cite{koch2013teichmuller} and explain why Theorems A, B and C apply
	to her construction. 
	It starts with a topological object on the 2-sphere $S^2$:
		
	\begin{defi}
		A topological polynomial is an orientation-preserving branched cover of the 2-sphere $h: S^2 \to S^2$, 
		such that there exists $p \in S^2$ with $h^{-1}(p)=\{p\}$, and such that the post-critical set $P_h$ is finite.
	\end{defi}

	Post-critically finite polynomials are topological polynomials in the sense of the previous definition,
	by identifying the Riemann sphere with $S^2$ and taking $p=\infty$.

	Given a topological polynomial $h$ and a complex structure $\struct$ on $(S^2, P_h)$,
	one call pull back the complex structure by $h$ to obtain a (possibly different) complex structure
	${h^*} {\struct}$. This descends to a holomorphic map $\sigma_h: \teich(S^2, P_h) \to \teich(S^2, P_h)$
	(see \cite{koch2013teichmuller} for more details). This map is of particular interest as it will have have 
	a fixed point in $\teich(S^2, P_h)$ if and only if $h$ is equivalent in some sense to an "usual" polynomial on the 
	Riemann sphere,
	by a celebrated theorem of Thurston.
	
	The \emph{moduli space} 
	$\mcal_{P_h}$ is the set of all injections $P_h \hookrightarrow \rs$ modulo composition by Möbius transformation.
	If $N=\card P_h$, it is easy to see that it identifies to an open subset of $\C^{N-3}$ if $N > 3$, and a single point
	otherwise. In the following, we will always assume that $N>3$, and let $k:=N-3$. 
	The Teichmüller space $\teich(S^2, P_h)$ is the universal cover of the moduli space: 
	there is a natural holomorphic universal covering $\pi : \teich(S^2, P_h) \to \mcal_{P_h}$.
	
	There are several interesting compactifications of the moduli space; in the case of a topological polynomial, 
	there is a distinguished hypersurface corresponding to the "infinity" point $p$. Therefore it is natural to 
	consider the compactification $\mcal_{P_h} \hookrightarrow \pk$, where the distinguished hypersurface would be
	in the hyperplane at infinity.
	
	In \cite{koch2013teichmuller}, Koch studies the existence of a holomorphic endomorphism $f: \pk \to \pk$ making the following 
	diagram commute:
	
	$$\xymatrix{
		\teich(S^2, P_h) \ar[r]^{\sigma_h}  \ar[d]_\pi & \teich(S^2, P_h)  \ar[d]^\pi \\
		\pk  & \pk \ar[l]^f
		}$$
	
	When it exists, such an endomorphism is called a \emph{moduli space map}. It is always post-critically finite, 
	and its post-critical set consists in a union of hyperplanes. Moreover, $\pk \backslash \pf(f) \simeq \mcal_{P_h}$ is 
	Kobayashi hyperbolic (Corollary 6.3).
	Koch proved that if $h$ is a topological polynomial that is either is unicritical or has a periodic critical point,
	then there exists a corresponding moduli space map (Theorem 5.17 and Theorem 5.18).

	Since any intersection of projective hyperplanes is a projective space, such intersections are always smooth.
	Moreover, they intersect each other weakly transversally in the sense of definition \ref{def:weaklytransverse}.
	
	Therefore, Theorems A, B and C apply to moduli space maps.

		\bibliographystyle{alpha}
		\bibliography{biblio}

\end{document}